\documentclass[12pt, twoside, leqno]{article}

\usepackage{amsmath,amsthm}
\usepackage{amssymb}

\usepackage{enumitem}

\usepackage{graphicx}

\usepackage[T1]{fontenc}

\newtheorem{theorem}{Theorem}[section]
\newtheorem{corollary}[theorem]{Corollary}
\newtheorem{lemma}[theorem]{Lemma}
\newtheorem{proposition}[theorem]{Proposition}

\theoremstyle{definition}
\newtheorem{definition}[theorem]{Definition}
\newtheorem{remark}[theorem]{Remark}

\numberwithin{equation}{section}

\allowdisplaybreaks[4]

\begin{document}


\baselineskip=17pt


\title{Ramanujan expansions of arithmetic functions of several variables over $\mathbb{F}_{q}[T]$}

\author{Tianfang Qi\\
Department of Mathematics\\
South China University of Technology\\
Guangzhou 510640
 Guangdong, China\\
E-mail: 15914449424@163.com
\and
Su Hu\\
Department of Mathematics\\
South China University of Technology\\
Guangzhou 510640
 Guangdong, China\\
E-mail: mahusu@scut.edu.cn}

\date{}

\maketitle


\renewcommand{\thefootnote}{}



\renewcommand{\thefootnote}{\arabic{footnote}}
\setcounter{footnote}{0}


\begin{abstract}
Let $\mathbb{A}=\mathbb{F}_{q}[T]$ be the  polynomial ring over finite field $\mathbb{F}_{q}$, and $\mathbb{A}_{+}$ be the set of monic polynomials in $\mathbb{A}$. In this paper, we show that a large
class of arithmetic functions in multi-variables over $\mathbb{A}$ can be expanded through the polynomial Ramanujan sums and the unitary polynomial Ramanujan sums. These are analogues of classical results over $\mathbb{N}$ by Delange, Ushiroya and T\'{o}th.\end{abstract}

\section{Introduction}
Let $\mathbb{N}$ be the set of positive integers and $(k,q)$ be the greatest common divisor of $k$ and $q$. The Ramanujan sums were first introduced by Ramanujan as follows
\begin{equation}\label{Ramanujan}
  c_{q}(n)=\mathop\sum_{\substack{k~\textrm{mod}~q\\(k,q)=1}}e^{\frac{2\pi{i}kn}{q}}\quad(q,n\in\mathbb{N}).
\end{equation}
 They satisfy many nice properties (see \cite[Sec. 8.3]{Apostol}).
Basing on the results given by Wintner \cite{Wintner}, Delange \cite{Delange} proved the following theorem, which shows that the arithmetic functions of one variable can be expanded through the Ramanujan sums (\ref{Ramanujan}).
\begin{theorem}[Delange \cite{Delange}]\label{Theorem 1}
Let $f:\mathbb{N}\rightarrow\mathbb{C}$ be an arithmetic function. Assume that
 \begin{equation*}
 \mathop\sum_{n=1}^{\infty}2^{\omega(n)}\frac{\left|(\mu\ast{f})(n)\right|}
 {n}<\infty.
 \end{equation*}
 Then for every $n\in\mathbb{N}$ we have the absolutely convergent Ramanujan expansion
 \begin{equation*}
 f(n)=\mathop\sum_{q=1}^{\infty}a_{q}c_{q}(n),
 \end{equation*}
 where the coefficients $a_{q}$ are given by
 \begin{equation*}
 a_{q}=\mathop\sum_{m=1}^{\infty}
 \frac{(\mu\ast{f})(mq)}{mq}\quad(q\in\mathbb{N}).
 \end{equation*}
\end{theorem}

In 2016, Ushiroya \cite{Ushiroya} extended Theorem \ref{Theorem 1} to arithmetic functions of two variables, and also get Ramanujan expansion of certain special functions.

We say that $d$ is a unitary divisor of $n$, denoted by $d||n$, if $d|n$, $(d,n/d)=1$.
Cohen \cite{Cohen2}  introduced the following unitary Ramanujan sums (also see T\'{o}th \cite[Sect 3.2]{T1})
\begin{equation}\label{unitary Ramanujan sums}
  c_{q}^{\ast}(n)=\mathop\sum_{\substack{k~\textrm{mod}~q\\(k,q)_{\ast}=1}}e^{\frac{2\pi{i}kn}{q}}\quad(q,n\in\mathbb{N}),
\end{equation}
where $(k,q)_{\ast}=\mathrm{max}\{d:d|k,~d||q\}$.

For any fixed $k\in\mathbb{N}$, if $f,g:\mathbb{N}^{k}\rightarrow\mathbb{C}$ are two arithmetic functions, then their Dirichlet convolution is
defined  by (see T\'{o}th \cite[16]{T1})
\begin{equation}\label{Dirichlet convolution 1}
  (f\ast{g})(n_{1},\cdots,n_{k})=\mathop\sum_{d_{1}|n_{1},\cdots,d_{k}|n_{k}}f(d_{1},\cdots,d_{k})g\left(\frac{n_{1}}{d_{1}},\cdots,\frac{n_{k}}{d_{k}}\right).
\end{equation}

Recently,  T\'{o}th \cite{T1} generalized the work of Delange and Ushiroya, and proved that the multi-variable arithmetic functions can be expanded through the Ramanujan sums (\ref{Ramanujan}) and the unitary Ramanujan sums (\ref{unitary Ramanujan sums}) introduced above. His result is as follows. 
\begin{theorem}[T\'{o}th {\cite[Theorem 2]{T1}}]\label{Thm,t1}
Let $f:\mathbb{N}^k\rightarrow\mathbb{C}$ be an arithmetic function $(k\in\mathbb{N})$. Assume that
 \begin{equation}
 \mathop\sum_{n_1,\cdots,n_k=1}^{\infty}2^{\omega(n_1)+\cdots+\omega(n_k)}\frac{\left|(\mu_{k}\ast{f})(n_1,\cdots,n_k)\right|}
 {{n_1}\cdots{n_k}}<\infty.
 \end{equation}
 Then, for every $n_1,\cdots,n_k\in\mathbb{N}$, we have
 \begin{align}\label{a}
 f(n_1,\cdots,n_k)=\mathop\sum_{q_1,\cdots,q_k=1}^{\infty}a_{q_1,\cdots,q_k}c_{q_{1}}(n_{1})
 \cdots{c_{q_{k}}(n_{k})},
 \end{align}
 \begin{align}\label{b}
 f(n_1,\cdots,n_k)=\mathop\sum_{q_1,\cdots,q_k=1}^{\infty}a_{q_1,\cdots,q_k}^{\ast}c_{q_{1}}^{\ast}(n_{1})
 \cdots{c_{q_{k}}^{\ast}(n_{k})},
 \end{align}
 where
 \begin{align*}
   &a_{q_1,\cdots,q_k}=\mathop\sum_{m_1,\cdots,m_k=1}^{\infty}
 \frac{(\mu_{k}\ast{f})(m_1q_1,\cdots,m_kq_k)}{{m_1q_1}\cdots{m_kq_k}},\\
 &a_{q_1,\cdots,q_k}^{\ast}=\mathop\sum_{\substack{m_1,\cdots,m_k=1\\(m_1,q_1)=1,\cdots,(m_k,q_k)=1}}^{\infty}
 \frac{(\mu_{k}\ast{f})(m_1q_1,\cdots,m_kq_k)}{{m_1q_1}\cdots{m_kq_k}},
 \end{align*}
 the series (\ref{a}) and (\ref{b}) being absolutely convergent.
 \end{theorem}

Let $(n_1,\cdots,n_k)$ be the greatest common divisor of $n_1,\cdots,n_k\in\mathbb{N}$ and $\sigma(n)$ and $\tau(n)$ be the sum of divisors of $n$ and the number of divisors of $n$, respectively.
Identities concerning some special functions can be obtained by setting $f(n_1,\cdots,n_k)=g((n_1,\cdots,n_k))$ in Theorem \ref{Thm,t1} (see T\'{o}th {\cite[Theorem 3]{T1}}), for example,
\begin{equation}\label{sigma*,n}
\frac{\sigma{((n_1,\cdots,n_k))}}{(n_1,\cdots,n_k)}
 =\zeta(k+1)\mathop\sum_{q_1,\cdots,q_k=1}^{\infty}
 \frac{\phi_{k+1}(Q)c_{q_1}^{\ast}(n_1)\cdots{c}_{q_k}^{\ast}(n_k)}{Q^{2(k+1)}}~(k\geq{1}),
 \end{equation}
 \begin{equation}\label{tau*,n}
 \tau{((n_1,\cdots,n_k))}
 =\zeta(k)\mathop\sum_{q_1,\cdots,q_k=1}^{\infty}
 \frac{\phi_{k}(Q)c_{q_1}^{\ast}(n_1)\cdots{c}_{q_k}^{\ast}(n_k)}{Q^{2k}}~(k\geq{2}),
 \end{equation}
 where $\zeta(k)$ are the special values of Riemann zeta function $\zeta(s)$ at positive integers $k$
 (see \cite[Eqs. (30) and (31)]{T1}).

Let $p$ be a prime number, $\mathbb{F}_{q}$ be the finite field with $q=p^{n}$ elements, $\mathbb{A}=\mathbb{F}_{q}[T]$ be the  polynomial ring over $\mathbb{F}_{q}$, and $\mathbb{A}_{+}$ be the set of monic polynomials in $\mathbb{A}$.
 It is well-known that there is an analogue on the arithmetic objects between $\mathbb{Z}$ and  $\mathbb{A}$. In analogue with the classical Ramanujan sums (\ref{Ramanujan}), Carlitz~\cite{Carlitz} first introduced the concept of polynomial Ramanujan sums over $\mathbb{A}$, and it has been generalized by Cohen in \cite{Cohen1}.

 Here we just briefly introduce some definitions and the main results of this paper. The notations involved will be explained in more detail in Sec.\ref{Pre}.
\begin{definition} [Carlitz~ {\cite[4.1]{Carlitz}} or Zheng~{\cite[1.10]{Zheng}}]
 For any $G,H\in\mathbb{A}$, the polynomial Ramanujan sum modulo $H$, denoted by $\eta{(G,H)}$, is given by
\begin{equation}\label{polynomial Ramanujan sums}
\eta{(G,H)}=\mathop\sum_{\substack{D~\mathrm{mod}~H\\(D,H)=1}}E(G,H)(D),
\end{equation}
where $(D,H)$ is the greatest  common monic divisor of $D$ and $H$ .
\end{definition}
We now give the definition of the unitary polynomial Ramanujan sums over $\mathbb{A}$,  which is an analogue of (\ref{unitary Ramanujan sums}).
\begin{definition}
For any $G,H\in\mathbb{A}$, the unitary polynomial Ramanujan sums over $\mathbb{A}$, denoted by $\eta^\ast(G,H)$, is given by
 \begin{equation}\label{def.up}
\eta^\ast(G,H)=\mathop\sum_{\substack{D~\mathrm{mod}~H\\(D,H)_{\ast}=1}}E(G,H)(D).
\end{equation}
\end{definition}

In order to demonstrate that there is a similar Ramanujan expansions for the arithmetic functions of multi-variables in the polynomial case, it needs to extend
the definition of arithmetic functions in one variable over $\mathbb{A}$ by Zheng (see {\cite[Definition 2.1]{Zheng}})  to the multi-variables .
\begin{definition} \label{Arith.}
 We say $f:(\mathbb{A}-\{0\})^{k} \rightarrow \mathbb{C}$ is a $k~(k\in\mathbb{N})$ variables arithmetic function if it is nonzero and for any
 $U_{1},\cdots,U_{k}\in\mathbb{F}_{q}^{\ast}$, $M_{1}, \cdots, M_{k}\in \mathbb{A}$, $f(U_{1}M_{1},\cdots,U_{k}M_{k})=f(M_{1},\cdots,M_{k})$.
\end{definition}
From the above definition, a function from $\mathbb{A}^{k}$ to $\mathbb{C}$ can be written as a function from $(\mathbb{A}_{+})^{k}$ to $\mathbb{C}$. Thus it is only need to consider the functions defined on $(\mathbb{A}_{+})^{k}$.

The following definition of Dirichlet convolution in the polynomial case is an analogue of (\ref{Dirichlet convolution 1}) given by T\'{o}th \cite{T1}.
\begin{definition}
For any fixed $k\in\mathbb{N}$, let $f,g:(\mathbb{A}_{+})^{k}\rightarrow\mathbb{C}$ be two arithmetic functions and the Dirichlet convolution is
defined by the following formula (all polynomials involved are assumed to be monic)
\begin{equation}\label{Dirichlet convolution.po}
  (f\ast{g})(G_{1},\cdots,G_{k})=\mathop\sum_{D_{1}|G_{1},\cdots,D_{k}|G_{k}}f(D_{1},\cdots,D_{k})g\left(\frac{G_{1}}{D_{1}},\cdots,\frac{G_{k}}{D_{k}}\right).
\end{equation}
\end{definition}

The following is the main result of this paper, which shows that the arithmetic functions of multi-variables over $\mathbb{A} $ can be expanded through the polynomial Ramanujan sums (\ref{polynomial Ramanujan sums}) and the unitary polynomial Ramanujan sums (\ref{def.up}).
\begin{theorem}\label{thm:general}
Let $f:(\mathbb{A}_{+})^k\rightarrow\mathbb{C}$ be an arithmetic function with $k\in\mathbb{N}$. Assume that
 \begin{equation}\label{value.general}
 \mathop\sum_{G_1,\cdots,G_k\in\mathbb{A}_{+}}2^{\omega(G_1)+\cdots+\omega(G_k)}\frac{\left|(\mu_{k}\ast{f})(G_1,\cdots,G_k)\right|}
 {\left|{G_1}\right|\cdots\left|{G_k}\right|}<\infty.
 \end{equation}
 Then, for any $G_1,\cdots,G_k\in\mathbb{A}_{+}$, we have
 \begin{equation}\label{po.ge.ab}
 f(G_1,\cdots,G_k)=\mathop\sum_{H_1,\cdots,H_k\in\mathbb{A}_{+}}\mathcal{C}_{H_1,\cdots,H_k}\eta(G_1,H_1)
 \cdots\eta(G_k,H_k),
 \end{equation}
 and
\begin{equation}\label{po.ge.ab*}
 f(G_1,\cdots,G_k)=\mathop\sum_{H_1,\cdots,H_k\in\mathbb{A}_{+}}\mathcal{C}_{H_1,\cdots,H_k}^{\ast}\eta^{\ast}(G_1,H_1)
 \cdots\eta^{\ast}(G_k,H_k),
 \end{equation}
 where
 \begin{equation}\label{po.ge.cof}
 \begin{split}
 &\mathcal{C}_{H_1,\cdots,H_k}=\mathop\sum_{M_1,\cdots,M_k\in\mathbb{A}_{+}}
 \frac{(\mu_{k}\ast{f})(M_1H_1,\cdots,M_kH_k)}{\left|{M_1H_1}\right|\cdots\left|{M_kH_k}\right|},\\
 &\mathcal{C}_{H_1,\cdots,H_k}^{\ast}=\mathop\sum_{\substack{M_1,\cdots,M_k\in\mathbb{A}_{+}\\(M_1,H_1)=1,\cdots,(M_k,H_k)=1}}
 \frac{(\mu_{k}\ast{f})(M_1H_1,\cdots,M_kH_k)}{\left|{M_1H_1}\right|\cdots\left|{M_kH_k}\right|}.
 \end{split}
 \end{equation}
The Ramanujan expansions (\ref{po.ge.ab}), (\ref{po.ge.ab*}) are all absolutely convergent.
 \end{theorem}

 Let $(G_1,\cdots,G_k)$ and $[G_1,\cdots,G_k]$ be the greatest common divisor and the least common multiple of $G_1,\cdots,G_k\in\mathbb{A}$, respectively. Let $g$ be an arithmetic function from $\mathbb{A}_{+}$ to $\mathbb{C}$. Setting $f(G_1,\cdots,G_k)=g((G_1,\cdots,G_k))$ in Theorem \ref{thm:general}, we obtain the following result.
\begin{theorem}\label{thm:special}
 Let ${g:\mathbb{A}_{+}\rightarrow\mathbb{C}}$ be an arithmetic function, $k\in\mathbb{N}$. Assume that
 \begin{equation}\label{value.special}
 {\mathop\sum_{G\in\mathbb{A}_{+}}2^{k\omega{(G)}}\frac{\left|{(\mu\ast{g})(G)}\right|}{\left|{G}\right|^{k}}}<\infty.
 \end{equation}
 Then for any ${G_1,\cdots,G_k\in\mathbb{A}_{+}}$, there exists absolutely convergent series
 \begin{equation}\label{po.sp.ab}
 g((G_1,\cdots,G_k))=\mathop\sum_{H_1,\cdots,H_k\in\mathbb{A}_{+}}\mathcal{C}_{H_1,\cdots,H_k}\eta(G_1,H_1)
 \cdots\eta(G_k,H_k),
 \end{equation}
 and
 \begin{equation}\label{po.sp.ab*}
 g((G_1,\cdots,G_k))=\mathop\sum_{H_1,\cdots,H_k\in\mathbb{A}_{+}}\mathcal{C}_{H_1,\cdots,H_k}^{\ast}\eta^{\ast}(G_1,H_1)
 \cdots\eta^{\ast}(G_k,H_k),
 \end{equation}
 where
 \begin{equation}\label{po.sp.coef}
 \begin{split}
   &\mathcal{C}_{H_1,\cdots,H_k}=\frac{1}{\left|{Q}\right|^{k}}\mathop\sum_{M\in\mathbb{A}_{+}}
 \frac{(\mu\ast{g})(MQ)}{\left|{M}\right|^{k}},\\
   &\mathcal{C}_{H_1,\cdots,H_k}^{\ast}=\frac{1}{\left|{Q}\right|^{k}}\mathop\sum_{\substack{M\in\mathbb{A}_{+}\\(M,Q)=1}}
 \frac{(\mu\ast{g})(MQ)}{\left|{M}\right|^{k}},
 \end{split}
 \end{equation}
 with $Q:=[H_1,\cdots,H_k]$.
 \end{theorem}

\begin{definition}[Rosen {\cite[p.15]{Rosen}}]\label{Divisor} For $G\in\mathbb{A}$, let  $\tau(G)$ be  the number of monic divisors of $G$ and $\sigma(G)=\sum_{D \mid G}\left|D\right|$, where the sum is taken over all monic divisors of $G$. \end{definition}

As in the classical case \cite{T1}, several identities concerning the arithmetic functions over $\Bbb{A}$ can also be derived from Theorem \ref{thm:special}, see corollaries in Sec. \ref{Main results}.
For example, we have
\begin{equation}
 \frac{\sigma{((G_1,\cdots,G_k))}}{\left|(G_1,\cdots,G_k)\right|}
 =\zeta_{\mathbb{A}}(k+1)\mathop\sum_{H_1,\cdots,H_k\in\mathbb{A}_{+}}
 \frac{\phi_{k+1}(Q)\eta^{\ast}(G_1,H_1)\cdots\eta^{\ast}(G_k,H_k)}{\left|Q\right|^{2(k+1)}},
 \end{equation}
 for any $k\geq{1}$, and
 \begin{equation}
 \tau{((G_1,\cdots,G_k))}
 =\zeta_{\mathbb{A}}(k)\mathop\sum_{H_1,\cdots,H_k\in\mathbb{A}_{+}}
 \frac{\phi_{k}(Q)\eta^{\ast}(G_1,H_1)\cdots\eta^{\ast}(G_k,H_k)}{\left|Q\right|^{2k}}~(k\geq{2}),
 \end{equation}
 where $$\zeta_{\mathbb{A}}(s)=\sum_{f\in\mathbb{A}_{+}}\frac{1}{|f|^{s}} $$
 is the zeta function of $\mathbb{A}=\mathbb{F}_{q}[T]$ (see \cite[p.11]{Rosen}).

It is worth to mention that although we mainly follow  the approach of T\'{o}th \cite{T1}, our proof will rely on the definitions of several new terminologies over
$\mathbb{A}$, such as the unitary polynomial Ramanujan sums, the arithmetic functions over unitary divisors and the convolution of multi--variable functions. In addition, recent results by Zheng \cite{Zheng}
on additive characters and polynomial Ramanujan sums on $\mathbb{A}$ are also very crucial to our study. To our purpose, we also need to prove some further properties on the (unitary) polynomial Ramanujan sums over $\Bbb{A}$, including their multiplicativities.

\section{Preliminaries}\label{Pre}

In what follows, $\mathbb{P}$ denotes the set of monic irreducible polynomials in $\mathbb{A}$. A divisor of $G\in\mathbb{A}$ always represents a monic divisor. For any $G\in\mathbb{A}$, $\left|G\right|=q^{\textrm{deg}G}$.

\subsection{Arithmetic functions of $k$ variables in the polynomial case}

The definition of Arithmetic functions of $k$ variables in the polynomial case has been given in Definition \ref{Arith.}.

An arithmetic function $f$ is called multiplicative if $f$ is nonzero, and
for any $G_{1},\cdots,G_{k},H_{1},\cdots,H_{k}\in\mathbb{A}_{+}$ with $(G_{1}\cdots{G_{k}},H_{1}\cdots{H_{k}})=1$,
we have $f(G_{1}H_{1},\cdots,G_{k}H_{k})=f(G_{1},\cdots,G_{k})f(H_{1},\cdots,H_{k})$.

It is easy to see that if $f$ is multiplicative, then its value is completely determined by $f(P^{e_{1}},\cdots,P^{e_{k}})$, where $P\in\mathbb{P}$, $e_{1},\cdots,e_{k}\geq{0}$.

The M\"{o}bius function $\mu$ on $\mathbb{A} $ is defined as follows.
\begin{definition}[Rosen {\cite[p.15]{Rosen}} or Zheng {\cite[p.872, Eq.(2.11)]{Zheng}}]\label{Mobius-de} For $G\in\mathbb{A}$, let $\mu(G)$ be 0 if $G$ is not square-free, and $(-1)^{t}$ if $G$ is a constant times a product of $t$ distinct monic irreducible polynomials.
\end{definition}

We also give definitions  for some other arithmetic functions on $\mathbb{A}$.

\noindent(\uppercase\expandafter{\romannumeral1}) Define $\sigma_{s}$ $(s\in\Bbb{R})$ by \begin{equation}\label{sigmas} \sigma_{s}(G)=\mathop\sum_{D|G}\left|{D}\right|^{s},\end{equation}
and it is easy to see that $\tau(G)=\sigma_{0}(G)$ (see Definition \ref{Divisor}).

\noindent(\uppercase\expandafter{\romannumeral2}) The Jordan totient function $\phi_{s}$ $(s\in\Bbb{R})$ is given by (see \cite[Eq. (2.20)]{Zheng}) \begin{equation}\label{phi} \phi_{s}{(G)}=\mathop\sum_{D|G}\mu{\left(\frac{G}{D}\right)}\left|{D}\right|^{s}=|G|^{s}\prod_{P|G}\left(1-\frac{1}{|P|^{s}}\right).\end{equation}

\noindent(\uppercase\expandafter{\romannumeral3}) Define $\beta_{s}$ $(s\in\Bbb{R})$ by
        \begin{equation}\label{beta} \beta_{s}(G)=\mathop\sum_{D|G}\left|{D}\right|^{s}\lambda{\left({\frac{G}{D}}\right)},\end{equation}
where \begin{equation}\label{lambda} \lambda(G)=(-1)^{\Omega(G)}~\textrm{and}~\Omega(G)=\sum_{P}\nu_{P}(G).\end{equation}

\noindent(\uppercase\expandafter{\romannumeral4}) Define $\psi_{s}$ $(s\in\Bbb{R})$ by
     \begin{equation}\label{psi} \psi_{s}(G)=|G|^{s}\prod_{P|G}\left(1+\frac{1}{|P|^{s}}\right).\end{equation}

\noindent(\uppercase\expandafter{\romannumeral5}) Define the arithmetic function  $\delta_{k}$ by
\begin{equation}\label{delta}
  \delta_{k}(G_{1},\cdots,G_{k})=\left\{
 \begin{array}{ll}
 1,&\textrm{if}~{G_{1},\cdots,G_{k}\in\mathbb{F}_{q}^{\ast}},\\
 0,&\textrm{otherwise}.
 \end{array}
 \right.
\end{equation}
It is easy to check that $\delta_{k}$ is the unity in the sense of the Dirichlet convolution (\ref{Dirichlet convolution.po}).

\noindent(\uppercase\expandafter{\romannumeral6}) Define the arithmetic function $\mu_{k}$ by
\begin{equation}\label{mu}
  \mu_{k}(G_{1},\cdots,G_{k})=\mu(G_{1})\cdots\mu(G_{k}),
\end{equation}
where $\mu$ is the M\"{o}bius function on $\mathbb{A}$. It is also easy to see that $\mu_{k}$ is the inverse of the constant $1$ function under (\ref{Dirichlet convolution.po}).

\subsection{On the definition of polynomial Ramanujan sums}

We now recall the notations appeared in the definition of the polynomial Ramanujan sums (\ref{polynomial Ramanujan sums}).

Assume that $H$ $(H\neq{0})$ is a fixed polynomial of degree $m$ in $\mathbb{A}$ and let $$A\equiv{a_{m-1}T^{m-1}+\dots+a_{1}T+a_{0}}~(\textrm{mod}~H),$$ then we have an additive function modulo $H$ on $\mathbb{A}$ defined by
 $$ t(A)=a_{m-1},~~\textrm{for}~\textrm{any}~A\in\mathbb{A}.$$
For any $A,B\in\mathbb{A} $, we have
$$ t(A+B)=t(A)+t(B),~ t(A)=t(B)~~ \textrm{if}~~A\equiv{B} ~(\textrm{mod}\;H),$$
in particular, $t(A)=0$ if $H|A$.

Then for any given $G\in\mathbb{A}$, let $t_{G}(A)=t(GA)$, we can easily check $t_{G}(A)$ is also an additive function modulo $H$.

Let $E(G,H)(A)=\lambda(t_{G}(A))$, where $\lambda(a)=e^{2\pi{i}\left(\frac{\textrm{tr}(a)}{p}\right)} (a\in\mathbb{F}_{q})$, $\textrm{tr}(a)$ is the trace map from $\mathbb{F}_{q}$ to $\mathbb{F}_{p}$, so $E(G,H)$ is an additive character modulo $H$ on $\mathbb{A} $, and it satisfies the orthogonal relation formula (see Zheng \cite[Lemma 2.3]{Zheng})
\begin{equation}\label{the orthogonal relation formula}
\mathop\sum_{G~\textrm{mod}~H}{E(G,H)(A)}=\left\{
 \begin{array}{ll}
\left|{H}\right|,&\mathrm{if}~{H|A},\\
0,&{\mathrm{otherwise}}.
 \end{array}
 \right.
 \end{equation}

From the above notations, we immediately reach the definition of the polynomial Ramanujan sums which has already been given in (\ref{polynomial Ramanujan sums}) above,
and we also have the following relevant formulas (see \cite[(1.12), (1.13), (1.14)]{Zheng} or \cite[(3.4), (3.5)]{Cohen1}),
\begin{equation}
  \eta(G_{1},H)=\eta(G_{2},H),\quad{\textrm{if}~G_{1}\equiv{G_{2}}~(\textrm{mod}~H)},
\end{equation}
\begin{equation}\label{mult.p}
  \eta(G,H_{1}H_{2})=\eta(G,H_{1})\eta(G,H_{2}),\quad{\textrm{if}~(H_{1},H_{2})=1},
\end{equation}
\begin{equation}\label{ide}
\eta{(G,H)}=\mathop\sum_{D|(G,H)}\left|{D}\right|\mu{\left(\frac{H}{D}\right)}.
\end{equation}

\subsection{Arithmetic functions defined by unitary divisors}

For $G\in\mathbb{A}$, $D$ $(D\in\mathbb{A}_{+})$ is a unitary divisor of $G$, denoted by $D||G$, if $D|G$ and $(D,G/D)=1$.
In the following definition, we extend the arithmetic functions defined by unitary divisors to the polynomial case.
\begin{definition}\label{Arithmetic}
For any $G\in\mathbb{A}$, let
\begin{equation*}
  \begin{split}
     \sigma^{\ast}{(G)}&=\mathop\sum_{D||G}\left|{D}\right|,~\tau^{\ast}(G)=\#\{D:D||G\},\\
     \phi^\ast(G)&=\#\{H:\mathrm{deg}H<\mathrm{deg}G,~(H,G)_{\ast}=1\},
   \end{split}
\end{equation*}
where $(H,G)_{\ast}=\mathop{\max}\limits_{\mathrm{deg}}\{D:D|H,D||G\}$, that is, $(H,G)_{\ast}$ is the element of the highest degree in $\{D:D|H,D||G\}$.
\end{definition}
They are analogous of $\sigma(G)$, $\tau(G)$ (see Definition \ref{Divisor}) and Euler's $\phi$ function $\phi(G)$ on $\mathbb{A}$, respectively.

We can easily see that the functions $\sigma^{\ast}, \tau^{\ast}, \phi^\ast$ are all multiplicative
and for any $P\in\mathbb{P}$ and $e\geq{1}$, $\sigma^{\ast}(P^{e})=\left|P\right|^{e}+1, \tau^{\ast}(P^{e})=2, \phi^\ast(P^{e})=\left|P\right|^{e}-1$.

\begin{definition}\label{Arithmetic}
For any $G\in\mathbb{A}$, let $\mu^{\ast}(G)=(-1)^{\omega(G)}$,
where $\omega(G)$ be the number of distinct monic irreducible divisors of $G$.
\end{definition}
 This is another analogue of the M\"{o}bius function (comparing with Definition \ref{Mobius-de} above).

 It's easy to see that it is multiplicative and it is also the inverse of the constant function $1$ under the unitary convolution of the functions $f$ and $g$ defined by
\begin{equation*}
  (f\times{g})(G)=\mathop\sum_{D||G}f(D)g\left(\frac{G}{D}\right),~~ \textrm{for~all}~G\in\mathbb{A}_{+}.
\end{equation*}

\subsection{Properties of (unitary) polynomial Ramanujan sums}

In this subsection, we prove several properties of the polynomial Ramanujan sums $\eta(G,H)$ and the unitary polynomial Ramanujan sums $\eta^\ast(G,H)$.
\begin{proposition}
For $G,H\in\mathbb{A}$,
\begin{equation}
\eta^\ast(G,H)=\mathop\sum_{D||(G,H)_{\ast}}\left|{D}\right|\mu^{\ast}{\left(\frac{H}{D}\right)}.
\end{equation}
\end{proposition}
\begin{proof}
 According to (\ref{ide}), the identity follows by comparing the definition of $\eta(G,H)$ and $\eta^\ast(G,H)$.
\end{proof}

From (\ref{mult.p}), we see that $\eta(G,H)$ is multiplicative, but we also have a simple proof here.
\begin{proposition}
For any $G\in\mathbb{A}$, $\eta(G,H)$ and $\eta^\ast(G,H)$ are multiplicative as functions of $H\in\mathbb{A}$.
\end{proposition}
\begin{proof}
 For any $H_{1},H_{2}\in\mathbb{A}$ and $(H_{1},H_{2})=1$,
since $E(G,H_{1}H_{2})$ is an additive character modulo $H_{1}H_{2}$, from the Chinese remainder theorem, we have
\begin{equation*}
\begin{split}
\eta{(G,H_{1}H_{2})}&=\mathop\sum_{\substack{D~\textrm{mod}~H_{1}H_{2}\\(D,H_{1}H_{2})=1}}E(G,H_{1}H_{2})(D)\\
&=\mathop\sum_{\substack{D~\textrm{mod}~H_{1}\\(D,H_{1})=1}}E(G,H_{1})(D)\mathop\sum_{\substack{D~\textrm{mod}~H_{2}\\(D,H_{2})=1}}E(G,H_{2})(D)\\
&=\eta(G,H_{1})\eta(G,H_{2}).
\end{split}
\end{equation*}
Therefore, $\eta(G,H)$ is multiplicative in $H$ for any fixed $G$.

Since $(D,H)=1$ implies $(D,H)_{\ast}=1$, the proof of the multiplicativity of $\eta^\ast(G,H)$ is along the same lines.
\end{proof}

Thus $\eta(G,H)$ and $\eta^\ast(G,H)$ are completely determined by its values on prime (monic irreducible polynomial) powers.
Using the multiplicativity of $\eta(G,H)$ and $\eta^\ast(G,H)$, we have the following proposition.
\begin{proposition}
For $G,H\in\mathbb{A}$,
\begin{equation}\label{equality*.1}
\mathop\sum_{D||H}\eta^\ast(G,D)=\left\{
 \begin{array}{ll}
 \left|{H}\right|,&\mathrm{if}~{H|G},\\
 0,&{\mathrm{otherwise}}.
 \end{array}
 \right.
\end{equation}
 \end{proposition}
\begin{proof}
We first prove the following identity
\begin{equation}
 \eta^\ast(G,P^{e})=\left\{
 \begin{array}{ll}
 \left|P\right|^{e}-1,&\mathrm{if}~{P^{e}|G},\\
 -1,&{\mathrm{otherwise}}.
 \end{array}
 \right.
\end{equation}
In fact, if $P^{e}|G$, then $E(G,P^{e})(D)=1$, and we have
\begin{equation*}
  \eta^\ast(G,P^{e})=\mathop\sum_{\substack{D~\textrm{mod}~P^{e}\\(D,P^{e})_{\ast}=1}}E(G,P^{e})(D)
  =\mathop\sum_{\substack{D~\textrm{mod}~P^{e}\\(D,P^{e})_{\ast}=1}}1
  =\phi^\ast(P^{e})
  =\left|P\right|^{e}-1.
\end{equation*}
If $P^{e}\nmid{G}$, then using the orthogonal relation formula (\ref{the orthogonal relation formula}), we have
\begin{equation*}
\begin{split}
  \eta^\ast(G,P^{e})
   &=\mathop\sum_{\substack{D~\textrm{mod}~P^{e}\\(D,P^{e})_{\ast}=1}}E(G,P^{e})(D)\\
   &=\mathop\sum_{D~\textrm{mod}~P^{e}}E(G,P^{e})(D)-\mathop\sum_{\substack{D~\textrm{mod}~P^{e}\\D\equiv{0}~(\textrm{mod}~P^{e})}}E(G,P^{e})(D)\\
   &=\mathop\sum_{D~\textrm{mod}~P^{e}}E(D,P^{e})(G)-1=-1.
\end{split}
\end{equation*}

Finally, let $H=P_{1}^{e_{1}}\cdots{P_{s}^{e_{s}}}$ be the prime decomposition of $H\in\mathbb{A}$. By (\ref{equality*.1}), we have
\begin{align*}
  \mathop\sum_{D||H}\eta^\ast{(G,D)}&=\mathop\sum_{D_{1}\cdots{D_{s}}||P_{1}^{e_{1}}\cdots{P_{s}^{e_{s}}}}\eta^\ast(G,D_{1}\cdots{D_{s}})\\
  &=\mathop\sum_{D_{1}||P_{1}^{e_{1}}}\cdots\mathop\sum_{D_{s}||P_{s}^{e_{s}}}\eta^\ast(G,D_{1}\cdots{D_{s}})\\
   &=\mathop\sum_{D_{1}||P_{1}^{e_{1}}}\eta^\ast(G,D_{1})\cdots\mathop\sum_{D_{s}||P_{s}^{e_{s}}}\eta^\ast(G,D_{s})\\
    &=\big(\eta^\ast(G,1)+\eta^\ast(G,P_{1}^{e_{1}})\big)\cdots\big(\eta^\ast(G,1)+\eta^\ast(G,P_{s}^{e_{s}})\big)\\
     &=\left\{
 \begin{array}{ll}
 \left|P_{1}\right|^{e_{1}},&\textrm{if}\, {P_{1}^{e_{1}}|G},\\
 0,&{\textrm{otherwise}}.
 \end{array}
 \right.\cdots\left\{
 \begin{array}{ll}
 \left|P_{s}\right|^{e_{s}},&\textrm{if}~{P_{s}^{e_{s}}|G},\\
 0,&{\textrm{otherwise}}.
 \end{array}
 \right.\\
     &=\left\{
 \begin{array}{ll}
 \left|H\right|,&\textrm{if}~{H|G},\\
 0,&{\textrm{otherwise}}.
 \end{array}
 \right.
\end{align*}
\end{proof}

The corresponding property for the polynomial Ramanujan sums $\eta(G,H)$ is as follows.
\begin{proposition}[Zheng {\cite[Corollary 4.3]{Zheng}}]
For $G,H\in\mathbb{A}$,
\begin{equation}\label{equality.1}
\mathop\sum_{D|H}\eta(G,D)=\left\{
 \begin{array}{ll}
 \left|{H}\right|,&\mathrm{if}~{H|G},\\
 0,&{\mathrm{otherwise}}.
 \end{array}
 \right.
 \end{equation}
 \end{proposition}

\begin{proposition}
For any $G,H\in\mathbb{A}$, we have
\begin{equation}\label{35}
  \mathop\sum_{D||H}\left|\eta^\ast(G,D)\right|=2^{\omega{\left(\frac{H}{(G,H)_{\ast}}\right)}}\left|{(G,H)_{\ast}}\right|,
\end{equation}
\begin{equation}\label{inequality* 1}
  \mathop\sum_{D||H}\left|\eta^\ast(G,D)\right|\leq{2^{\omega{(H)}}\left|G\right|},
\end{equation}
\begin{equation}\label{inequality 2}
  \mathop\sum_{D|H}\left|\eta(G,D)\right|\leq{2^{\omega{(H)}}\left|G\right|}.
\end{equation}
\end{proposition}

\begin{proof}
First, we consider the case of the unitary polynomial Ramanujan sums $\eta^\ast(G,H)$. Let $H=P_{1}^{e_{1}}\cdots{P_{s}^{e_{s}}}$ be the prime decomposition of $H\in\mathbb{A}$. Using its multiplicativity and identity (\ref{equality*.1}), we have
\begin{align*}
  \mathop\sum_{D||H}\left|\eta^\ast{(G,D)}\right|
  &=\mathop\sum_{D_{1}\cdots{D_{s}}||P_{1}^{e_{1}}\cdots{P_{s}^{e_{s}}}}
    \left|\eta^\ast(G,D_{1}\cdots{D_{s}})\right|\\
  &=\mathop\sum_{D_{1}||P_{1}^{e_{1}}}\cdots\mathop\sum_{D_{s}||P_{s}^{e_{s}}}\left|\eta^\ast(G,D_{1}\cdots{D_{s}})\right|\\
  &=\mathop\sum_{D_{1}||P_{1}^{e_{1}}}\left|\eta^\ast(G,D_{1})\right|\cdots\mathop\sum_{D_{s}||P_{s}^{e_{s}}}\left|\eta^\ast(G,D_{s})\right|\\
  &=\left(\left|\eta^\ast(G,1)\right|+\left|\eta^\ast(G,P_{1}^{e_{1}})\right|\right)\cdots
    \left(\left|\eta^\ast(G,1)\right|+\left|\eta^\ast(G,P_{s}^{e_{s}})\right|\right)\\
  &=\left(\left\{
 \begin{array}{ll}
 \left|P_{1}\right|^{e_{1}},&\textrm{if}~{P_{1}^{e_{1}}|G},\\
 2,&{\textrm{otherwise}},
 \end{array}
 \right.\right)\cdots\left(\left\{
 \begin{array}{ll}
 \left|P_{s}\right|^{e_{s}},&\textrm{if}~{P_{s}^{e_{s}}|G},\\
 2,&{\textrm{otherwise}},
 \end{array}
 \right.\right)\\
     &=2^{\omega{\left(\frac{H}{(G,H)_{\ast}}\right)}}\left|{(G,H)_{\ast}}\right|.
\end{align*}
It is (\ref{35}). From (\ref{35}), we immediately obtain (\ref{inequality* 1}).

Next, we consider the case of the polynomial Ramanujan sums $\eta(G,H)$ and the method presented here is analogous to that given by Delange \cite{Delange}.
Let
\begin{equation*}
\begin{split}
   \mathop\sum_{D|P^{e}}\left|{\eta(G,D)}\right|&=\mathop\sum_{i=0}^{e}\left|{\eta(G,P^{i})}\right|\\
  &=\left|{\eta(G,1)}\right|+\left|{\eta(G,P)}\right|+\dots+\left|{\eta(G,P^{e})}\right|.
\end{split}
\end{equation*}

If $P\nmid{G}$ and $s\geq{1}$, then we have
\begin{equation*}
  \eta(G,P^{s})=\mathop\sum_{D|(G,P^{s})}\left|D\right|\mu\left(\frac{P^{s}}{D}\right)
  =\mu(P^{s}),
\end{equation*}
so by the definition of $\mu$, we obtain
\begin{equation*}
  \mathop\sum_{D|P^{e}}\left|{\eta(G,D)}\right|=\left|{\eta(G,1)}\right|+\left|{\eta(G,P)}\right|
  =\left|\mu(1)\right|+\left|\mu(P)\right|=2.
\end{equation*}

If $P^{a}||G$ and $a\geq{1}$, then we have
\begin{equation}\label{4.1}
\begin{split}
 \eta(G,P^{e})&=\mathop\sum_{i=0}^{\textrm{min}\{a,e\}}\left|P^{i}\right|\mu(P^{e-i})\\
  &=\left\{
 \begin{array}{ll}
 \left|P\right|^{e}-\left|P\right|^{e-1},&\textrm{if}~{1\leq{e}\leq{a}},\\
 -\left|P\right|^{a},&\textrm{if}~{e=a+1},\\
 0,&{\textrm{otherwise}}.
 \end{array}
 \right.
\end{split}
\end{equation}

\noindent $\textrm{i})$ If $1\leq{e}\leq{a}$, then by (\ref{4.1}), we have
\begin{equation*}
  \mathop\sum_{D|P^{e}}\left|{\eta(G,D)}\right|=1+(\left|P\right|-1)+\dots+(\left|P\right|^{e}-\left|P\right|^{e-1})=\left|P\right|^{e}.
\end{equation*}

\noindent $\textrm{ii})$ If $e>a$, then by (\ref{4.1}), we have
\begin{align*}
  \mathop\sum_{D|P^{e}}\left|{\eta(G,D)}\right|
  &=\left|{\eta(G,1)}\right|+\left|{\eta(G,P)}\right|+\dots+\left|{\eta(G,P^{a})}\right|+\left|{\eta(G,P^{a+1})}\right|\\
  &=1+(\left|P\right|-1)+\dots+(\left|P\right|^{a}-\left|P\right|^{a-1})+\left|({-\left|P\right|^{a}})\right|\\
  &=2\left|P\right|^{a}.
\end{align*}

Combining $\textrm{i})$ and $\textrm{ii})$, we get
\begin{equation*}
  \mathop\sum_{D|P^{e}}\left|{\eta(G,D)}\right|=\left\{
 \begin{array}{ll}
 \left|P\right|^{e},&\textrm{if}~{1\leq{e}\leq{a}},\\
 2\left|P\right|^{a},&\textrm{if}~{e>a},
 \end{array}
 \right.
\end{equation*}
which implies that
\begin{equation*}
  \mathop\sum_{D|P^{e}}\left|{\eta(G,D)}\right|\leq{2\left|P\right|^{\nu_{P}(G)}}.
\end{equation*}
Since $\eta(G,H)$ is multiplicative, we have
\begin{equation*}
\begin{split}
  \mathop\sum_{D|H}\left|{\eta(G,D)}\right|&=\prod_{P|H}\mathop\sum_{D|P^{e}}\left|{\eta(G,D)}\right|\\
  &\leq{\left(\prod_{\substack{P|H\\P\nmid{G}}}2\right)}\left(\prod_{P|(H,G)}2\left|P\right|^{\nu_{P}(G)}\right)\\
   &\leq{2^{\omega(H)}}\prod_{P|(H,G)}\left|P\right|^{\nu_{P}(G)}\\
   &\leq{2^{\omega(H)}}\prod_{P|G}\left|P\right|^{\nu_{P}(G)}\\
  &\leq2^{\omega(H)}\left|G\right|,
\end{split}
\end{equation*}
which is the desired result.
\end{proof}

\section{Proofs of the Main results and their corollaries}\label{Main results}
In this section, we shall proof our main results (Theorems \ref{thm:general} and  \ref{thm:special}), and we will also show some of their corollaries. 

First, we give a proof for the general result (Theorem \ref{thm:general}).

\noindent\textbf{Proof of Theorem \ref{thm:general}}
First we prove (\ref{po.ge.ab*}) concerning the expansion through unitary Ramanujan sums $\eta^\ast(G,H)$.

From equality (\ref{equality*.1}), we have
\begin{align*}
  &f(G_1,\cdots,G_k)\\
 &=\mathop\sum_{D_1|G_1,\cdots,D_k|G_k}{(\mu_k\ast{f})(D_1,\cdots,D_k)}\\
 &=\mathop\sum_{D_1|G_1,\cdots,D_k|G_k}\frac{(\mu_k\ast{f})(D_1,\cdots,D_k)}
 {\left|{D_1}\right|\cdots\left|{D_k}\right|}\left|{D_1}\right|\cdots\left|{D_k}\right|\\
 &=\mathop\sum_{D_1|G_1,\cdots,D_k|G_k}\frac{(\mu_k\ast{f})(D_1,\cdots,D_k)}
 {\left|{D_1}\right|\cdots\left|{D_k}\right|}
 \mathop\sum_{H_1\parallel{D_1}}\eta^\ast(G_1,H_1)\cdots\mathop\sum_{H_k\parallel{D_k}}\eta^\ast(G_k,H_k)\\
 &=\mathop\sum_{H_1,\cdots,H_k\in\mathbb{A}_{+}}\eta^\ast(G_1,H_1)\cdots\eta^\ast(G_k,H_k)
 \mathop\sum_{\substack{D_1,\cdots,D_k\in\mathbb{A}_{+}\\H_1\parallel{D_1},\cdots,H_k\parallel{D_k}}}
 \frac{(\mu_k\ast{f})(D_1,\cdots,D_k)}
 {\left|{D_1}\right|\cdots\left|{D_k}\right|}\\
 &=\mathop\sum_{H_1,\cdots,H_k\in\mathbb{A}_{+}}\eta^\ast(G_1,H_1)\cdots\eta^\ast(G_k,H_k)\\
  &\quad\quad\times\mathop\sum_{\substack{M_1,\cdots,M_k\in\mathbb{A}_{+}\\(M_1,H_1)=1,\cdots,(M_k,H_k)=1}}
 \frac{(\mu_{k}\ast{f})(M_1H_1,\cdots,M_kH_k)}{\left|{M_1H_1}\right|\cdots\left|{M_kH_k}\right|}.
\end{align*}
Thus (\ref{po.ge.ab*}) is obtained if we denote
 \[
 \mathcal{C}^{\ast}_{H_1,\cdots,H_k}=\mathop\sum_{\substack{M_1,\cdots,M_k\in\mathbb{A}_{+}\\(M_1,H_1)=1,\cdots,(M_k,H_k)=1}}
 \frac{(\mu_{k}\ast{f})(M_1H_1,\cdots,M_kH_k)}{\left|{M_1H_1}\right|\cdots\left|{M_kH_k}\right|}.
 \]

 In the following, we prove that the series (\ref{po.ge.ab*}) is absolutely convergent. By  inequality (\ref{inequality* 1}), we have
 \begin{align*}
   &\left|f(G_1,\cdots,G_k)\right|\\
 &=\left|\mathop\sum_{H_1,\cdots,H_k\in\mathbb{A}_{+}}\mathcal{C}_{H_1,\cdots,H_k}^{\ast}\eta^{\ast}(G_1,H_1)
   \dots\eta^{\ast}(G_k,H_k)\right|\\
 &\leq\mathop\sum_{H_1,\cdots,H_k\in\mathbb{A}_{+}}\left|\mathcal{C}_{H_1,\cdots,H_k}^{\ast}\right|\left|\eta^{\ast}(G_1,H_1)\right|
   \cdots\left|\eta^{\ast}(G_k,H_k)\right|\\
 &\leq\mathop\sum_{H_1,\cdots,H_k\in\mathbb{A}_{+}}\mathop\sum_{\substack{M_1,\cdots,M_k\in\mathbb{A}_{+}\\(M_1,H_1)=1,\cdots,(M_k,H_k)=1}}
   \frac{\left|(\mu_{k}\ast{f})(M_1H_1,\cdots,M_kH_k)\right|}{\left|{M_1H_1}\right|\cdots\left|{M_kH_k}\right|}
   \left|\eta^{\ast}(G_1,H_1)\right|\cdots\left|\eta^{\ast}(G_k,H_k)\right|\\
 &=\mathop\sum_{T_1,\cdots,T_k\in\mathbb{A}_{+}}\frac{\left|(\mu_{k}\ast{f})(T_1,\dots,T_k)\right|}{\left|{T_1}\right|\cdots\left|{T_k}\right|}
   \mathop\sum_{\substack{M_1H_1=T_1\\(M_1,H_1)=1}}\left|\eta^{\ast}(G_1,H_1)\right|\cdots
   \mathop\sum_{\substack{M_kH_k=T_1\\(M_k,H_k)=1}}\left|\eta^{\ast}(G_k,H_k)\right|\\
 &=\mathop\sum_{T_1,\cdots,T_k\in\mathbb{A}_{+}}\frac{\left|(\mu_{k}\ast{f})(T_1,\cdots,T_k)\right|}{\left|{T_1}\right|\cdots\left|{T_k}\right|}
   \mathop\sum_{{H_1}||{T_1}}\left|\eta^{\ast}(G_1,H_1)\right|\cdots
   \mathop\sum_{{H_k}||{T_k}}\left|\eta^{\ast}(G_k,H_k)\right|\\
 &\leq\mathop\sum_{T_1,\cdots,T_k\in\mathbb{A}_{+}}\frac{\left|(\mu_{k}\ast{f})(T_1,\cdots,T_k)\right|}{\left|{T_1}\right|\cdots\left|{T_k}\right|}
   (2^{\omega(T_1)}\left|{G_1}\right|)\cdots(2^{\omega(T_k)}\left|{G_k}\right|)\\
 &=\left|{G_1}\right|\cdots\left|{G_k}\right|\mathop\sum_{T_1,\cdots,T_k\in\mathbb{A}_{+}}
   2^{\omega(T_1)+\cdots+\omega(T_k)}\frac{\left|(\mu_{k}\ast{f})(T_1,\cdots,T_k)\right|}
   {\left|{T_1}\right|\cdots\left|{T_k}\right|}<\infty,
 \end{align*}
which is the desired result.

 The rest part of the proof of  (\ref{po.ge.ab}) concerning the Ramanujan sums $\eta(G,H)$ is along the same lines as the above by using equality (\ref{equality.1}) and inequality (\ref{inequality 2}).
We completes the proof of Theorem \ref{thm:general}.

\begin{lemma}If $f$ is multiplicative, then (\ref{value.general}) equals to
 \begin{equation}\label{value.ge.1}
 \mathop\sum_{G_1,\cdots,G_k\in\mathbb{A}_{+}}\frac{\left|(\mu_{k}\ast{f})(G_1,\cdots,G_k)\right|}
 {\left|{G_1}\right|\cdots\left|{G_k}\right|}<\infty,
 \end{equation}
 and
 \begin{equation}\label{value.ge.2}
 \prod_{P\in\mathbb{P}}\mathop\sum_{\substack{e_1,\cdots,e_k=0\\e_1+\cdots+e_k\geq{0}}}^{\infty}\frac{\left|(\mu_{k}\ast{f})(P^{e_1},\cdots,P^{e_k})\right|}
 {\left|P\right|^{e_1+\cdots+e_k}}<\infty.
 \end{equation}
\end{lemma}

\begin{proof}
It's clear that (\ref{value.ge.1}) equals to (\ref{value.ge.2}) by the multiplicativity of $f$, so we just need to prove (\ref{value.general}) equals to (\ref{value.ge.1}).

In fact, for any $G_1,\cdots,G_k\in\mathbb{A}_{+}$, $\omega(G_i)\geq 0$ $(1\leq i\leq k)$, we have
$$\frac{\left|(\mu_{k}\ast{f})(G_1,\cdots,G_k)\right|}{\left|{G_1}\right|\cdots\left|{G_k}\right|}
 \leq 2^{\omega(G_1)+\cdots+\omega(G_k)}\frac{\left|(\mu_{k}\ast{f})(G_1,\cdots,G_k)\right|}
 {\left|{G_1}\right|\cdots\left|{G_k}\right|}.$$
 It follows that (\ref{value.general}) implies (\ref{value.ge.1}).

 Conversly, for any $G_1,\cdots,G_k\in\mathbb{A}_{+}$, $0\leq\omega(G_i)\leq 1$ $(1\leq i\leq k)$, this implies $2^{\omega(G_1)+\cdots+\omega(G_k)}$ has upper bound $2^{k}$.
 Then (\ref{value.general}) follows from (\ref{value.ge.1}) by Abel criterion for the convergence of series.
\end{proof}
From the above lemma, we have the following corollary.
 \begin{corollary}\label{co.ge}
 Let $f:(\mathbb{A}_{+})^k\rightarrow\mathbb{C}$ be a multiplicative function with $k\in\mathbb{N}$. Assume that $(\ref{value.ge.1})$ or $(\ref{value.ge.2})$ holds, then for any $G_1,\cdots,G_k\in\mathbb{A}_{+}$, we have $(\ref{po.ge.ab})$ and $(\ref{po.ge.ab*})$, which are absolutely convergent series, and the coefficients can be written as follows:
 \begin{equation*}
 \mathcal{C}_{H_1,\cdots,H_k}=\prod_{P\in\mathbb{P}}\mathop\sum_{e_1\geq\nu_{p}{(H_1)},\cdots,e_k\geq\nu_{p}{(H_k)}}
 \frac{(\mu_{k}\ast{f})(P^{e_1},\cdots,P^{e_k})}{\left|P\right|^{e_1+\cdots+e_k}},
 \end{equation*}

 \begin{equation*}
 \mathcal{C}_{H_1,\cdots,H_k}^{\ast}=\prod_{P\in\mathbb{P}}{\mathop\sum^{'}_{e_1,\cdots,e_k}}
 \frac{(\mu_{k}\ast{f})(P^{e_1},\cdots,P^{e_k})}{\left|P\right|^{e_1+\cdots+e_k}}.
 \end{equation*}
The notation $\mathop\sum^{'}_{e_1,\cdots,e_k}$ means that for fixed $P$ and $i$,
if $\nu_{P}{(H_i)}=0$, then in the summation, $e_i$ is taken over all the integers satisfying $\geq 0$, and if $\nu_{P}{(H_i)}\geq{1}$,
then in the summation, $e_i$ is taken over all the integers satisfying $\geq 1$.
  \end{corollary}
\begin{proof}
  By Theorem \ref{thm:general}, we have the identity
 \begin{equation*}
   \mathcal{C}_{H_1,\cdots,H_k}=\mathop\sum_{M_1,\cdots,M_k\in\mathbb{A}_{+}}
 \frac{(\mu_{k}\ast{f})(M_1H_1,\cdots,M_kH_k)}{\left|{M_1H_1}\right|\cdots\left|{M_kH_k}\right|}.
 \end{equation*}
  If $f$ is multiplicative, then $\mu\ast{f}$ is also multiplicative. Since $\nu_{P}(M_{i}H_{i})\geq\nu_{P}(H_i)$ for any $1\leq{i}\leq{k}$ and $P\in\mathbb{P}$, from the multiplicativity
  of $\mu\ast{f}$, we have
 \begin{equation*}
 \begin{aligned}
   \mathcal{C}_{H_1,\cdots,H_k}
   &=\prod_{P\in\mathbb{P}}\mathop\sum_{\substack{M_1,\cdots,M_k\in\mathbb{A}_{+}\\\nu_{P}(M_{i}H_{i})\geq\nu_{P}(H_i)~(1\leq{i}\leq{k})}}
     \frac{(\mu_{k}\ast{f})(P^{\nu_{P}(M_{1}H_{1})},\cdots,P^{\nu_{P}(M_{k}H_{k})})}{|P|^{\nu_{P}(M_{1}H_{1})+\cdots+\nu_{P}(M_{k}H_{k})}}\\
   &=\prod_{P\in\mathbb{P}}\mathop\sum_{e_1\geq\nu_{p}{(H_1)},\cdots,e_k\geq\nu_{p}{(H_k)}}
     \frac{(\mu_{k}\ast{f})(P^{e_1},\cdots,P^{e_k})}{\left|P\right|^{e_1+\cdots+e_k}},
 \end{aligned}
 \end{equation*}
 where $e_{i}=\nu_{P}(M_{i}H_{i})$.

By Theorem \ref{thm:general}, we have
 \[
  \mathcal{C}_{H_1,\cdots,H_k}^{\ast}
 =\mathop\sum_{\substack{M_1,\cdots,M_k\in\mathbb{A}_{+}\\(M_1,H_1)=1,\cdots,(M_k,H_k)=1}}
  \frac{(\mu_{k}\ast{f})(M_1H_1,\cdots,M_kH_k)}{\left|{M_1H_1}\right|\cdots\left|{M_kH_k}\right|}.
 \]
 Notice that $(M_i,H_i)=1$ for $1\leq i\leq k$. For any fixed $P$ and $i$, if ${\nu_{p}{(H_i)}=0}$, that is ${(P,H_i)=1}$, then $\nu_{P}(M_{i})\geq{0}$, and
 $\nu_{P}(M_{i}H_{i})=\nu_{P}(M_{i})+\nu_{P}(H_{i})\geq{0}$.
If ${\nu_{p}{(H_i)}\geq{1}}$, that is ${P\mid{H_i}}$, then ${P\nmid {M_i}}$ and $\nu_{P}(M_{i})={0}$, it follows that
 $\nu_{P}(M_{i}H_{i})=\nu_{P}(M_{i})+\nu_{P}(H_{i})\geq{1}$.

 Finally, from the multiplicativity of $\mu\ast{f}$, we obtain
 \begin{equation*}
 \mathcal{C}_{H_1,\cdots,H_k}^{\ast}
 =\prod_{P\in\mathbb{P}}{\mathop\sum^{'}_{e_1,\cdots,e_k}}
  \frac{(\mu_{k}\ast{f})(P^{e_1},\cdots,P^{e_k})}{\left|P\right|^{e_1+\cdots+e_k}},
 \end{equation*}
 which is the desired result.
 \end{proof}

Setting $f(G_1,\cdots,G_k)=g((G_1,\cdots,G_k))$ in Theorem \ref{thm:general}, we have Theorem \ref{thm:special}. But we present a specific proof here.

\noindent\textbf{Proof of Theorem \ref{thm:special}:}
 Since
 \begin{equation*}
 \begin{split}
 g((G_1,\cdots,G_k))&=\mathop\sum_{D|(G_1,\cdots,G_k)}{(\mu\ast{g})(D)}\\
 &=\mathop\sum_{D|G_1,\cdots,D|G_k}{(\mu\ast{g})(D)}\\
 &=\mathop\sum_{D|G_1,\cdots,D|G_k}{(\mu\ast{g})((D,\cdots,D))}\\
 &=\mathop\sum_{D|G_1,\cdots,D|G_k}{(\mu_{k}\ast{f})(D,\cdots,D)},
 \end{split}
 \end{equation*}
from M\"{o}bius inversion formula, we have
\begin{equation*}
 (\mu_{k}\ast{f})(G,\cdots,G)=(\mu\ast{g})(G),
\end{equation*}
which follows that
\begin{equation}\label{20a}
(\mu_{k}\ast{f})(G_1,\cdots,G_k)=\left\{
 \begin{array}{ll}
 (\mu\ast{g})(G),&\textrm{if}~{G_1=\cdots=G_k=G},\\
 0,&{\textrm{otherwise}}.
 \end{array}
 \right.\end{equation}

 From $(\ref{po.ge.cof})$ and (\ref{20a}), we have
 \begin{align*}
    \mathcal{C}_{H_1,\cdots,H_k}
 &=\mathop\sum_{M_1,\cdots,M_k\in\mathbb{A}_{+}}
 \frac{(\mu_{k}\ast{f})(M_1H_1,\cdots,M_kH_k)}{\left|{M_1H_1}\right|\cdots\left|{M_kH_k}\right|}\\
 &=\mathop\sum_{\substack{G\in\mathbb{A}_{+}\\M_1H_1=\cdots=M_kH_k=G}}
 \frac{(\mu_{k}\ast{f})(G,\dots,G)}{\left|{G}\right|\cdots\left|{G}\right|}\\
 &=\mathop\sum_{\substack{G\in\mathbb{A}_{+}\\{H_1}|G,\cdots,{H_k}|G}}
 \frac{(\mu\ast{g})(G)}{\left|{G}\right|^{k}}\\
 &=\mathop\sum_{\substack{{G\in\mathbb{A}_{+}}\\{[H_1,\cdots,H_k]|G}}} \frac{(\mu\ast{g})(G)}{\left|{G}\right|^{k}}.
 \end{align*}
Then setting $Q=[H_1,\cdots,H_k]$ and $G=MQ$, we get
 \begin{equation*}
 \mathcal{C}_{H_1,\cdots,H_k}=\frac{1}{\left|{Q}\right|^{k}}\mathop\sum_{M\in\mathbb{A}_{+}}
 \frac{(\mu\ast{g})(MQ)}{\left|{M}\right|^{k}}.
 \end{equation*}

 From $(\ref{po.ge.cof})$ and (\ref{20a}), we have
 \begin{align*}
   \mathcal{C}_{H_1,\cdots,H_k}^{\ast}
 &=\mathop\sum_{\substack{M_1,\cdots,M_k\in\mathbb{A}_{+}\\(M_1,H_1)=1,\cdots,(M_k,H_k)=1}}
 \frac{(\mu_{k}\ast{f})(M_1H_1,\cdots,M_kH_k)}{\left|{M_1H_1}\right|\cdots\left|{M_kH_k}\right|}\\
 &=\mathop\sum_{\substack{G\in\mathbb{A}_{+}\\M_1H_1=\dots=M_kH_k=G\\(M_1,H_1)=1,\cdots,(M_k,H_k)=1}}
 \frac{(\mu_{k}\ast{f})(G,\cdots,G)}{\left|{G}\right|\cdots\left|{G}\right|}\\
 &=\mathop\sum_{\substack{G\in\mathbb{A}_{+}\\{H_1}||G,\cdots,{H_k}||G}}
 \frac{(\mu\ast{g})(G)}{\left|{G}\right|^{k}}\\
 &=\mathop\sum_{\substack{{G\in\mathbb{A}_{+}}\\{[H_1,\cdots,H_k]||G}}}\frac{(\mu\ast{g})(G)}{\left|{G}\right|^{k}}.
 \end{align*}
Then setting $Q=[H_1,\cdots,H_k]$ and $G=MQ$, we obtain
 \begin{equation*}
  \mathcal{C}_{H_1,\cdots,H_k}^{\ast}
 =\frac{1}{\left|{Q}\right|^{k}}\mathop\sum_{\substack{M\in\mathbb{A}_{+}\\(M,Q)=1}}
  \frac{(\mu\ast{g})(MQ)}{\left|{M}\right|^{k}}.
 \end{equation*}

 \begin{remark} If the function $g$ is multiplicative, then (\ref{value.special}) equals to
 \begin{equation}\label{value.sp.1}
 \mathop\sum_{G\in\mathbb{A}_{+}}\frac{\left|(\mu\ast{g})(G)\right|}
 {\left|{G}\right|^{k}}<\infty,
 \end{equation}
 and
 \begin{equation}\label{value.sp.2}
 \prod_{P\in\mathbb{P}}\mathop\sum_{e=1}^{\infty}\frac{\left|(\mu\ast{g})(P^{e})\right|}
 {\left|P\right|^{ke}}<\infty.
 \end{equation}
 \end{remark}

 We also have the following Corollary.
\begin{corollary}\label{important corollary}
 Let $g:\mathbb{A}_{+}\rightarrow\mathbb{C}$ be a multiplicative function, $k\in\mathbb{N}$. Assume that $(\ref{value.sp.1})$ or $(\ref{value.sp.2})$ holds.

 If ${\mu\ast{g}}$ is completely multiplicative, then for any $G_1,\cdots,G_k\in\mathbb{A}_{+}$, we have the absolutely convergent series $(\ref{po.sp.ab})$ and its coefficients satisfy
 \begin{equation}\label{42}
 \mathcal{C}_{H_1,\cdots,H_k}=\frac{(\mu\ast{g})(Q)}{\left|{Q}\right|^{k}}\mathop\sum_{M\in\mathbb{A}_{+}}
 \frac{(\mu\ast{g})(M)}{\left|{M}\right|^{k}}.
 \end{equation}

 If $g$ is multiplicative, then for any $G_1,\cdots,G_k\in\mathbb{A}_{+}$, we have the absolutely convergent series $(\ref{po.sp.ab*})$ and its coefficients satisfy
 \begin{equation}\label{43}
 \mathcal{C}_{H_1,\cdots,H_k}^{\ast}=\frac{(\mu\ast{g})(Q)}{\left|{Q}\right|^{k}}\mathop\sum_{\substack{M\in\mathbb{A}_{+}\\(M,Q)=1}}
 \frac{(\mu\ast{g})(M)}{\left|{M}\right|^{k}}.
 \end{equation}
 \end{corollary}

\begin{proof}
The proof is divided into two cases.

 If $\mu\ast{g}$ is completely multiplicative, then from (\ref{po.sp.coef}) we have
 \begin{align*}
   \mathcal{C}_{H_1,\cdots,H_k}
 &=\frac{1}{\left|{Q}\right|^{k}}\mathop\sum_{M\in\mathbb{A}_{+}}
 \frac{(\mu\ast{g})(MQ)}{\left|{M}\right|^{k}}\\
 &=\frac{1}{\left|{Q}\right|^{k}}\mathop\sum_{M\in\mathbb{A}_{+}}
 \frac{(\mu\ast{g})(Q)(\mu\ast{g})(M)}{\left|{M}\right|^{k}}\\
 &=\frac{(\mu\ast{g})(Q)}{\left|{Q}\right|^{k}}\mathop\sum_{M\in\mathbb{A}_{+}}
 \frac{(\mu\ast{g})(M)}{\left|{M}\right|^{k}}.
 \end{align*}

 If $g$ is multiplicative, then $\mu\ast{g}$ is multiplicative, and from (\ref{po.sp.coef}) we have
 \begin{align*}
    \mathcal{C}_{H_1,\cdots,H_k}^{\ast}&=\frac{1}{\left|{Q}\right|^{k}}\mathop\sum_{\substack{M\in\mathbb{A}_{+}\\(M,Q)=1}}
 \frac{(\mu\ast{g})(MQ)}{\left|{M}\right|^{k}}\\
 &=\frac{1}{\left|{Q}\right|^{k}}\mathop\sum_{\substack{M\in\mathbb{A}_{+}\\(M,Q)=1}}
 \frac{(\mu\ast{g})(Q)(\mu\ast{g})(M)}{\left|{M}\right|^{k}}\\
 &=\frac{(\mu\ast{g})(Q)}{\left|{Q}\right|^{k}}\mathop\sum_{\substack{M\in\mathbb{A}_{+}\\(M,Q)=1}}
 \frac{(\mu\ast{g})(M)}{\left|{M}\right|^{k}}.
 \end{align*}
 \end{proof}

From Corollary \ref{important corollary}, we can derive the following identities involving $$\zeta_{\mathbb{A}}(s)=\sum_{f\in\mathbb{A}_{+}}\frac{1}{|f|^{s}}, $$
 the zeta function of $\mathbb{A}$.
 \begin{corollary}\label{Co.sigma}
 Let the arithmetic function $\sigma_{s}$ be defined as in (\ref{sigmas}). For any $G_1,\cdots,G_k\in\mathbb{A}_{+}$, we have the following absolutely convergent series
 \begin{equation}
 \frac{\sigma_{s}{((G_1,\cdots,G_k))}}{\left|(G_1,\cdots,G_k)\right|^{s}}
 =\zeta_{\mathbb{A}}(k+s)\mathop\sum_{H_1,\cdots,H_k\in\mathbb{A}_{+}}
 \frac{\eta(G_1,H_1)\cdots\eta(G_k,H_k)}{\left|Q\right|^{k+s}},
 \end{equation}
where $s\in\mathbb{R},k+s>1$,
 \begin{equation}\label{sigma}
 \frac{\sigma{((G_1,\cdots,G_k))}}{\left|(G_1,\cdots,G_k)\right|}
 =\zeta_{\mathbb{A}}(k+1)\mathop\sum_{H_1,\cdots,H_k\in\mathbb{A}_{+}}
 \frac{\eta(G_1,H_1)\cdots\eta(G_k,H_k)}{\left|Q\right|^{k+1}}~(k\geq{1}),
 \end{equation}
 \begin{equation}\label{tau}
 \tau{((G_1,\cdots,G_k))}
 =\zeta_{\mathbb{A}}(k)\mathop\sum_{H_1,\cdots,H_k\in\mathbb{A}_{+}}
 \frac{\eta(G_1,H_1)\cdots\eta(G_k,H_k)}{\left|Q\right|^{k}}~(k\geq{2}).
 \end{equation}
 \end{corollary}

 \begin{corollary}\label{Co.sigma*}
 Let the arithmetic functions $\sigma_{s}$, $\phi_{s}$ be defined as in (\ref{sigmas}) and (\ref{phi}), respectively. For any $G_1,\cdots,G_k\in\mathbb{A}_{+}$, we have the following absolutely convergent series
 \begin{equation}
 \frac{\sigma_{s}{((G_1,\cdots,G_k))}}{\left|(G_1,\cdots,G_k)\right|^{s}}
 =\zeta_{\mathbb{A}}(k+s)\mathop\sum_{H_1,\cdots,H_k\in\mathbb{A}_{+}}
 \frac{\phi_{k+s}(Q)\eta^{\ast}(G_1,H_1)\cdots\eta^{\ast}(G_k,H_k)}{\left|Q\right|^{2(k+s)}},
 \end{equation}
 where $s\in\mathbb{R},k+s>{1}$,
 \begin{equation}\label{sigma*}
 \frac{\sigma{((G_1,\cdots,G_k))}}{\left|(G_1,\cdots,G_k)\right|}
 =\zeta_{\mathbb{A}}(k+1)\mathop\sum_{H_1,\cdots,H_k\in\mathbb{A}_{+}}
 \frac{\phi_{k+1}(Q)\eta^{\ast}(G_1,H_1)\cdots\eta^{\ast}(G_k,H_k)}{\left|Q\right|^{2(k+1)}},
 \end{equation}
 where $k\geq{1}$,
 \begin{equation}\label{tau*}
 \tau{((G_1,\cdots,G_k))}
 =\zeta_{\mathbb{A}}(k)\mathop\sum_{H_1,\cdots,H_k\in\mathbb{A}_{+}}
 \frac{\phi_{k}(Q)\eta^{\ast}(G_1,H_1)\cdots\eta^{\ast}(G_k,H_k)}{\left|Q\right|^{2k}}~(k\geq{2}).
 \end{equation}
 \end{corollary}

\noindent\textbf{Proofs of Corollary \ref{Co.sigma} and Corollary \ref{Co.sigma*}:}
Set
 \[
 g(G)=\frac{\sigma_{s}{(G)}}{\left|{G}\right|^{s}},
 \]
 where
 \[
 \sigma_{s}(G)=\mathop\sum_{D|G}\left|{D}\right|^{s}.
 \]
We have
 \begin{align}\label{star}
 g(G)=\frac{\sigma_{s}{(G)}}{\left|{G}\right|^{s}}
 =\frac{\mathop\sum_{D|G}\left|{D}\right|^{s}}{\left|{G}\right|^{s}}
 =\mathop\sum_{D|G}\frac{1}{\left|{\frac{G}{D}}\right|^{s}}.
 \end{align}
Then by M\"obius  inversion formula, we obtain
\begin{equation}\label{Mobius}
 (\mu\ast{g})(G)=\frac{1}{\left|{G}\right|^{s}}.
\end{equation}
It is easy to check that $g$ is multiplicative, and $\mu\ast{g}$ is completely multiplicative.

From (\ref{42}) and (\ref{Mobius}), we have
 \begin{equation*}
 \begin{split}
 \mathcal{C}_{H_1,\cdots,H_k}&=\frac{(\mu\ast{g})(Q)}{\left|{Q}\right|^{k}}\mathop\sum_{M\in\mathbb{A}_{+}}
 \frac{(\mu\ast{g})(M)}{\left|{M}\right|^{k}}\\
 &=\frac{1}{\left|{Q}\right|^{k+s}}\mathop\sum_{M\in\mathbb{A}_{+}}\frac{1}{\left|{M}\right|^{k+s}}\\
 &=\frac{\zeta_{\mathbb{A}}(k+s)}{\left|{Q}\right|^{k+s}}.
 \end{split}
 \end{equation*}
Hence applying (\ref{po.sp.ab}) to
 \[
 g(G)=\frac{\sigma_{s}{(G)}}{\left|{G}\right|^{s}},
 \] we get
 \begin{equation*}
\frac{\sigma_{s}{((G_1,\cdots,G_k))}}{\left|(G_1,\cdots,G_k)\right|^{s}}
 =\zeta_{\mathbb{A}}(k+s)\mathop\sum_{H_1,\cdots,H_k\in\mathbb{A}_{+}}
 \frac{\eta(G_1,H_1)\cdots\eta(G_k,H_k)}{\left|Q\right|^{k+s}}.
 \end{equation*}

Setting $s=1$ and $s=0$ in the above equality, we obtain (\ref{sigma}) and (\ref{tau}), respectively.

Now we prove (\ref{sigma*}) and (\ref{tau*}). From (\ref{43}) and (\ref{Mobius}), we have
\begin{align*}
 \mathcal{C}_{H_1,\cdots,H_k}^{\ast}
 &=\frac{(\mu\ast{g})(Q)}{\left|{Q}\right|^{k}}\mathop\sum_{\substack{M\in\mathbb{A}_{+}\\(M,Q)=1}}
 \frac{(\mu\ast{g})(M)}{\left|{M}\right|^{k}}\\
 &=\frac{1}{\left|{Q}\right|^{k+s}}\mathop\sum_{\substack{M\in\mathbb{A}_{+}\\(M,Q)=1}}\frac{1}{\left|{M}\right|^{k+s}}\\
 &=\frac{1}{\left|{Q}\right|^{k+s}}\prod_{\substack{P\in\mathbb{P}\\(P,Q)=1}}\left(1-\frac{1}{\left|{P}\right|^{k+s}}\right)^{-1}\\
 &=\frac{1}{\left|{Q}\right|^{k+s}}\frac{\prod_{P\in\mathbb{P}}\left(1-\frac{1}{\left|{P}\right|^{k+s}}\right)^{-1}}
 {\prod_{P|Q}\left({1-\frac{1}{\left|{P}\right|^{k+s}}}\right)^{-1}}\\
 &=\frac{\zeta_{\mathbb{A}}(k+s)}{\left|{Q}\right|^{2(k+s)}}\left|{Q}\right|^{k+s}\prod_{P|Q}\left(1-\frac{1}{\left|{P}\right|^{k+s}}\right)\\
 &=\zeta_{\mathbb{A}}(k+s)\frac{\phi_{k+s}(Q)}{\left|{Q}\right|^{2(k+s)}}.
\end{align*}
Hence applying (\ref{po.sp.ab*}) to
 \[
 g(G)=\frac{\sigma_{s}{(G)}}{\left|{G}\right|^{s}},
 \] we get
 \begin{equation*}
 \frac{\sigma_{s}{((G_1,\cdots,G_k))}}{\left|(G_1,\cdots,G_k)\right|^{s}}
 =\zeta_{\mathbb{A}}(k+s)\mathop\sum_{H_1,\cdots,H_k\in\mathbb{A}_{+}}
 \frac{\phi_{k+s}(Q)\eta^{\ast}(G_1,H_1)\cdots\eta^{\ast}(G_k,H_k)}{\left|Q\right|^{2(k+s)}}.
 \end{equation*}

Setting $s=1$ and $s=0$, we obtain (\ref{sigma*}) and (\ref{tau*}), respectively.
We completes the proofs of Corollary \ref{Co.sigma} and Corollary \ref{Co.sigma*}.

\begin{corollary}\label{coro5}
Let the arithmetic functions $\beta_{s}, \lambda$ and $\psi_{s}$ be defined as in (\ref{beta}), (\ref{lambda}) and (\ref{psi}), respectively.
For any  $G_1,\cdots,G_k\in\mathbb{A}_{+}$, $s\in\mathbb{R},s+k>1$, we have the following absolutely convergent series
 \begin{equation}
 \frac{\beta_{s}{((G_1,\cdots,G_k))}}{\left|(G_1,\cdots,G_k)\right|^{s}}
 =\frac{\zeta_{\mathbb{A}}(2(k+s))}{\zeta_{\mathbb{A}}(k+s)}\mathop\sum_{H_1,\cdots,H_k\in\mathbb{A}_{+}}
 \frac{\lambda{(Q)}\eta(G_1,H_1)\cdots\eta(G_k,H_k)}{\left|Q\right|^{k+s}},
 \end{equation}
 \begin{equation}\label{55}
 \frac{\beta_{s}{((G_1,\cdots,G_k))}}{\left|(G_1,\cdots,G_k)\right|^{s}}
 =\frac{\zeta_{\mathbb{A}}(2(k+s))}{\zeta_{\mathbb{A}}(k+s)}\mathop\sum_{H_1,\cdots,H_k\in\mathbb{A}_{+}}
 \frac{\lambda{(Q)}\psi_{k+s}(Q)\eta^\ast(G_1,H_1)\cdots\eta^\ast(G_k,H_k)}{\left|Q\right|^{2(k+s)}}.
 \end{equation}
 \end{corollary}

\begin{proof}
Set
 \[
 g(G)=\frac{\beta_{s}{(G)}}{\left|{G}\right|^{s}},
 \]
 where
 \[
 \beta_{s}(G)=\mathop\sum_{D|G}\left|{D}\right|^{s}\lambda{\left({\frac{G}{D}}\right)}
 \]
 and the inner function $\lambda(G)=(-1)^{\Omega(G)}$ with $\Omega(G)=\sum_{P}\nu_{P}(G)$.

 We have
 \begin{equation*}
 g(G)=\frac{\beta_{s}(G)}{\left|{G}\right|^{s}}
 =\frac{\mathop\sum_{D|G}\left|{D}\right|^{s}\lambda{\left({\frac{G}{D}}\right)}}{\left|{G}\right|^{s}}
 =\mathop\sum_{D|G}\frac{\lambda{\left({\frac{G}{D}}\right)}}{\left|{\frac{G}{D}}\right|^{s}}.
 \end{equation*}
Then by M\"obius  inversion formula, we obtain
\begin{equation}\label{54}
 (\mu\ast{g})(G)=\frac{\lambda{(G)}}{\left|{G}\right|^{s}}.
\end{equation}
It is easy to check that $g$ is multiplicative, and $\mu\ast{g}$ is completely multiplicative.

From (\ref{42}) and (\ref{54}), we have
\begin{align*}
  \mathcal{C}_{H_1,\cdots,H_k}&=\frac{(\mu\ast{g})(Q)}{\left|{Q}\right|^{k}}\mathop\sum_{M\in\mathbb{A}_{+}}
 \frac{(\mu\ast{g})(M)}{\left|{M}\right|^{k}}\\
 &=\frac{\lambda{(Q)}}{\left|{Q}\right|^{k+s}}\mathop\sum_{M\in\mathbb{A}_{+}}\frac{\lambda{(M)}}{\left|{M}\right|^{k+s}}\\
 &=\frac{\lambda{(Q)}}{\left|{Q}\right|^{k+s}}\prod_{P\in\mathbb{P}}\left(\mathop\sum_{e=0}^{\infty}
 \frac{\lambda{(P^{e})}}{\left|{P}\right|^{e(k+s)}}\right)\\
 &=\frac{\lambda{(Q)}}{\left|{Q}\right|^{k+s}}\prod_{P\in\mathbb{P}}\left(1-\frac{\lambda{(P)}}{\left|{P}\right|^{k+s}}\right)^{-1}\\
 &=\frac{\lambda{(Q)}}{\left|{Q}\right|^{k+s}}\prod_{P\in\mathbb{P}}\left(1+\frac{1}{\left|{P}\right|^{k+s}}\right)^{-1}\\
 &=\frac{\lambda{(Q)}}{\left|{Q}\right|^{k+s}}\frac{\prod_{P\in\mathbb{P}}\left(1-\frac{1}{\left|{P}\right|^{2(k+s)}}\right)^{-1}}
 {\prod_{P\in\mathbb{P}}\left(1-\frac{1}{\left|{P}\right|^{k+s}}\right)^{-1}}\\
 &=\frac{\zeta_{\mathbb{A}}(2(k+s))}{\zeta_{\mathbb{A}}(k+s)}{\frac{\lambda{(Q)}}{\left|{Q}\right|^{k+s}}}.
\end{align*}
Hence applying (\ref{po.sp.ab}) to
 \[
 g(G)=\frac{\beta_{s}{(G)}}{\left|{G}\right|^{s}},
 \]
 we get
 \[
 \frac{\beta_{s}{((G_1,\cdots,G_k))}}{\left|(G_1,\cdots,G_k)\right|^{s}}
 =\frac{\zeta_{\mathbb{A}}(2(k+s))}{\zeta_{\mathbb{A}}(k+s)}\mathop\sum_{H_1,\cdots,H_k\in\mathbb{A}_{+}}
 \frac{\lambda{(Q)}\eta(G_1,H_1)\cdots\eta(G_k,H_k)}{\left|{Q}\right|^{k+s}}.
 \]

Now we prove (\ref{55}). From (\ref{43}) and (\ref{54}), we have
\begin{align*}
  \mathcal{C}^{\ast}_{H_1,\cdots,H_k}&=\frac{(\mu\ast{g})(Q)}{\left|{Q}\right|^{k}}\mathop\sum_{\substack{M\in\mathbb{A}_{+}\\(M,Q)=1}}
 \frac{(\mu\ast{g})(M)}{\left|{M}\right|^{k}}\\
 &=\frac{\lambda{(Q)}}{\left|{Q}\right|^{k+s}}\mathop\sum_{\substack{M\in\mathbb{A}_{+}\\(M,Q)=1}}\frac{\lambda{(M)}}{\left|{M}\right|^{k+s}}\\
 &=\frac{\lambda{(Q)}}{\left|{Q}\right|^{k+s}}\prod_{\substack{P\in\mathbb{P}\\(P,Q)=1}}\left(\mathop\sum_{e=0}^{\infty}
 \frac{\lambda{(P^{e})}}{\left|{P}\right|^{e(k+s)}}\right)\\
 &=\frac{\lambda{(Q)}}{\left|{Q}\right|^{k+s}}\prod_{\substack{P\in\mathbb{P}\\(P,Q)=1}}\left(1-\frac{\lambda{(P)}}{\left|{P}\right|^{k+s}}\right)^{-1}\\
 &=\frac{\lambda{(Q)}}{\left|{Q}\right|^{k+s}}\prod_{\substack{P\in\mathbb{P}\\(P,Q)=1}}\left(1+\frac{1}{\left|{P}\right|^{k+s}}\right)^{-1}\\
 &=\frac{\lambda{(Q)}}{\left|{Q}\right|^{k+s}}\frac{\prod_{P\in\mathbb{P}}\left(1+\frac{1}{\left|{P}\right|^{k+s}}\right)^{-1}}
 {\prod_{P|Q}\left(1+\frac{1}{\left|{P}\right|^{k+s}}\right)^{-1}}\\
 &=\frac{\lambda{(Q)}}{\left|{Q}\right|^{2(k+s)}}\left({\left|{Q}\right|^{k+s} \prod_{P|Q}\left(1+\frac{1}{\left|{P}\right|^{k+s}}\right)}\right)
 \prod_{P\in\mathbb{P}}\left(1+\frac{1}{\left|{P}\right|^{k+s}}\right)^{-1}\\
 &=\frac{\zeta_{\mathbb{A}}(2(k+s))}{\zeta_{\mathbb{A}}(k+s)}
 \frac{\lambda{(Q)}\psi_{k+s}(Q)}{\left|{Q}\right|^{2(k+s)}}.
\end{align*}
Hence applying (\ref{po.sp.ab*}) to
 \[
 g(G)=\frac{\beta_{s}{(G)}}{\left|{G}\right|^{s}},
 \]  we get
  \[
 \frac{\beta_{s}{((G_1,\cdots,G_k))}}{\left|(G_1,\cdots,G_k)\right|^{s}}
 =\frac{\zeta_{\mathbb{A}}(2(k+s))}{\zeta_{\mathbb{A}}(k+s)}\mathop\sum_{H_1,\cdots,H_k\in\mathbb{A}_{+}}
 \frac{\lambda{(Q)}\psi_{k+s}(Q)\eta^\ast(G_1,H_1)\cdots\eta^\ast(G_k,H_k)}{\left|Q\right|^{2(k+s)}}.
 \]
 \end{proof}

 \begin{corollary}\label{coll.special}
 For any $G_1,\cdots,G_k\in\mathbb{A}_{+}$, $k\in\mathbb{N}$, we have the following absolutely convergent series
 \begin{equation}
 \frac{\phi_{s}{((G_1,\cdots,G_k))}}{\left|(G_1,\cdots,G_k)\right|^{s}}
 =\frac{1}{\zeta_{\mathbb{A}}(k+s)}\mathop\sum_{H_1,\cdots,H_k\in\mathbb{A}_{+}}
 \frac{\mu{(Q)}\eta^\ast(G_1,H_1)\cdots\eta^\ast(G_k,H_k)}{\phi_{k+s}(Q)},
 \end{equation}
 where $s\in\mathbb{R}$ with $k+s>1$,
 \begin{equation}\label{phi*}
 \frac{\phi{((G_1,\cdots,G_k))}}{\left|(G_1,\cdots,G_k)\right|}
 =\frac{1}{\zeta_{\mathbb{A}}(k+1)}\mathop\sum_{H_1,\cdots,H_k\in\mathbb{A}_{+}}
 \frac{\mu{(Q)}\eta^\ast(G_1,H_1)\cdots\eta^\ast(G_k,H_k)}{\phi_{k+1}(Q)},
 \end{equation}
 where $k\geq{1}$.
 \end{corollary}

 \begin{proof}
Set
 \[
 g(G)=\frac{\phi_{s}{(G)}}{\left|{G}\right|^{s}},
 \]
 where
 \[
 \phi_{s}{(G)}=\mathop\sum_{D|G}\mu{\left(\frac{G}{D}\right)}\left|{D}\right|^{s}.
 \]
 We have
 \begin{equation*}
   g(G)=\frac{\phi_{s}{(G)}}{\left|{G}\right|^{s}}=\frac{\mathop\sum_{D|G}\mu{\left(\frac{G}{D}\right)}\left|{D}\right|^{s}}{\left|{G}\right|^{s}}
   =\mathop\sum_{D|G}\frac{\mu{\left(\frac{G}{D}\right)}}{\left|\frac{G}{D}\right|^{s}}.
 \end{equation*}
  Then by M\"obius  inversion formula, we obtain
\begin{equation}\label{71}
 (\mu\ast{g})(G)=\frac{\mu{(G)}}{\left|{G}\right|^{s}}.
\end{equation}
It is easy to see that $g$ is multiplicative, but $\mu\ast{g}$ is not completely multiplicative in this case.

From (\ref{43}) and (\ref{71}), we have
\begin{align*}
  \mathcal{C}^{\ast}_{H_1,\cdots,H_k}&=\frac{(\mu\ast{g})(Q)}{\left|{Q}\right|^{k}}\mathop\sum_{\substack{M\in\mathbb{A}_{+}\\(M,Q)=1}}
 \frac{(\mu\ast{g})(M)}{\left|{M}\right|^{k}}\\
 &=\frac{\mu{(Q)}}{\left|{Q}\right|^{k+s}}\mathop\sum_{\substack{M\in\mathbb{A}_{+}\\(M,Q)=1}}\frac{\mu{(M)}}{\left|{M}\right|^{k+s}}\\
 &=\frac{\mu{(Q)}}{\left|{Q}\right|^{k+s}}\prod_{\substack{P\in\mathbb{P}\\(P,Q)=1}}\left(\mathop\sum_{e=0}^{\infty}
 \frac{\mu{(P^{e})}}{\left|{P}\right|^{e(k+s)}}\right)\;(e=0,1)\\
 &=\frac{\mu{(Q)}}{\left|{Q}\right|^{k+s}}\prod_{\substack{P\in\mathbb{P}\\(P,Q)=1}}\left({1-\frac{1}{\left|{P}\right|^{k+s}}}\right)\\
 &=\frac{\mu{(Q)}}{\left|{Q}\right|^{k+s}}\frac{\prod_{P\in\mathbb{P}}\left(1-\frac{1}{\left|{P}\right|^{k+s}}\right)}
 {\prod_{P|Q}\left(1-\frac{1}{\left|{P}\right|^{k+s}}\right)}\\
 &=\frac{\mu{(Q)}}{\left|{Q}\right|^{(k+s)}\prod_{P|Q}\left({1-\frac{1}{\left|{P}\right|^{k+s}}}\right)}
 \prod_{P\in\mathbb{P}}\left(1-\frac{1}{\left|{P}\right|^{k+s}}\right)\\
 &=\frac{1}{\zeta_{\mathbb{A}}{(k+s)}}\frac{\mu{(Q)}}{\phi_{k+s}{(Q)}}.
\end{align*}
Hence applying (\ref{po.sp.ab*}) to
 \[
 g(G)=\frac{\phi_{s}{(G)}}{\left|{G}\right|^{s}},
 \]
 we get
 \[
 \frac{\phi_{s}{((G_1,\cdots,G_k))}}{\left|(G_1,\cdots,G_k)\right|^{s}}
 =\frac{1}{\zeta_{\mathbb{A}}(k+s)}\mathop\sum_{H_1,\cdots,H_k\in\mathbb{A}_{+}}
 \frac{\mu{(Q)}\eta^\ast(G_1,H_1)\cdots\eta^\ast(G_k,H_k)}{\phi_{k+s}(Q)}.
 \]

Finally, setting $s=1$, we obtain (\ref{phi*}).
 \end{proof}

\begin{corollary}\label{Cor.last}
For any $G_1,\cdots,G_k\in\mathbb{A}_{+}$, we have the following absolutely convergent series
 \begin{equation}
 \frac{\phi_{s}{((G_1,\cdots,G_k))}}{\left|(G_1,\cdots,G_k)\right|^{s}}
 =\frac{1}{\zeta_{\mathbb{A}}(k+s)}\mathop\sum_{H_1,\cdots,H_k\in\mathbb{A}_{+}}
 \frac{\mu{(Q)}\eta(G_1,H_1)\cdots\eta(G_k,H_k)}{\phi_{k+s}(Q)},
 \end{equation}
 where $s\in\mathbb{R}$, $k+s>1$,
 \begin{equation}
 \frac{\phi{((G_1,\cdots,G_k))}}{\left|(G_1,\cdots,G_k)\right|}
 =\frac{1}{\zeta_{\mathbb{A}}(k+1)}\mathop\sum_{H_1,\cdots,H_k\in\mathbb{A}_{+}}
 \frac{\mu{(Q)}\eta(G_1,H_1)\cdots\eta(G_k,H_k)}{\phi_{k+1}(Q)},
 \end{equation}
 where $k\geq{1}$.
 \end{corollary}

 \begin{proof}
Set
 \[
 g(G)=\frac{\phi_{s}{(G)}}{\left|{G}\right|^{s}}.
 \]
From the proof of Corollary \ref{coll.special}, we have
\begin{equation}\label{61}
 (\mu\ast{g})(G)=\frac{\mu{(G)}}{\left|{G}\right|^{s}}.
\end{equation}
 Since here $\mu\ast{g}$ is not completely multiplicative,
 we can not apply (\ref{42}), but we may use (\ref{po.sp.coef}) to calculate the coefficients.

From (\ref{po.sp.coef}) and (\ref{61}), we have
 \begin{equation*}
 \begin{split}
 \mathcal{C}_{H_1,\cdots,H_k}&=\frac{1}{\left|{Q}\right|^{k}}\mathop\sum_{M\in\mathbb{A}_{+}}
 \frac{(\mu\ast{g})(MQ)}{\left|{M}\right|^{k}}\\
 &=\frac{1}{\left|{Q}\right|^{k+s}}\mathop\sum_{M\in\mathbb{A}_{+}}\frac{\mu{(MQ)}}{\left|{M}\right|^{k+s}}.
 \end{split}
 \end{equation*}
 Note that if $(M,Q)\neq{1}$, then $\mu{(MQ)}=0$. So the above equation becomes to
 \begin{equation*}
 \begin{split}
  \mathcal{C}_{H_1,\cdots,H_k}&=\frac{1}{\left|{Q}\right|^{k+s}}\mathop\sum_{\substack{M\in\mathbb{A}_{+}\\(M,Q)=1}}\frac{\mu{(MQ)}}{\left|{M}\right|^{k+s}}\\
 &=\frac{1}{\left|{Q}\right|^{k+s}}\mathop\sum_{\substack{M\in\mathbb{A}_{+}\\(M,Q)=1}}\frac{\mu{(M)}\mu{(Q)}}{\left|{M}\right|^{k+s}}\\
 &=\frac{\mu{(Q)}}{\left|{Q}\right|^{k+s}}\mathop\sum_{\substack{M\in\mathbb{A}_{+}\\(M,Q)=1}}\frac{\mu{(M)}}{\left|{M}\right|^{k+s}}.
 \end{split}
 \end{equation*}
The remaining part of the proof follows from the same line as that of Corollary \ref{coll.special}.
\end{proof}


\end{document}